\documentclass[11pt]{article}
\usepackage[utf8]{inputenc}
\usepackage{amsmath}
\usepackage{amsthm}
\usepackage{amssymb}
\usepackage{hyperref}
\usepackage{enumitem}
\setlist{itemsep=0pt}

\newtheorem{theorem}{Theorem}
\newtheorem{lemma}{Lemma}

\newcommand{\pL}{\mathrm{L}}

\newcommand{\SL}{\mathrm{SL}}
\newcommand{\GL}{\mathrm{GL}}

\newcommand{\F}{\mathbb{F}}

\newcommand{\N}{\mathbb{N}}

\newcommand{\C}{\mathbb{C}}

\newcommand{\Aut}{{\rm Aut}}
\newcommand{\Out}{\mathrm {Out}}

\newcommand{\Hom}{{\rm Hom}}
\newcommand{\End}{{\rm End}}

\renewcommand{\char}{{\rm char}}
\newcommand{\rC}{{\rm C}}
\newcommand{\rN}{{\rm N}}
\newcommand{\rR}{{\rm R}}
\newcommand{\rZ}{{\rm Z}}
\newcommand{\gk}{{\kappa}}

\begin{document}
\title{The soluble radical and orbits of certain maps on finite groups}
\author{David Popović, John S. Wilson \\
\\
In memoriam Jan Saxl}
\maketitle

\begin{abstract}
For each element $u$ in a finite group $G$ define a map $\theta_u\colon G\to G$ by $\theta_u(g)=[g^{-u},g]$ and set $\Theta_G(u)=\{g\in G\mid \theta_u^n(g)=g \hbox{ for some } n>0\}$.
Then $\theta_u$ induces a permutation of $\Theta_G(u)$; let $\beta_G(u)$ be the number of orbits apart from $\{1\}$.  Building on work of J.N.\ Bray, R.A.\ Wilson and the second author, we show that the index of the soluble radical of a finite group $G$ is bounded in terms of the values of $\beta_G(u)$ for $2$-elements $u$.	
%Let $G$ be a finite group.  For each $u\in G$ define $\theta_u\colon G\to G$ by $\theta_u(g)=[g^{-u},g]$ and $\Theta_G(u)=\{g\in G\mid \theta_u^n(g)=g \hbox{ for some } n>0\}$.
%Then $\theta_u$ induces a permutation of $\Theta_G(u)$; let $\beta_G(u)$ be the number of orbits apart from $\{1\}$.  Building on work of J.N.\ Bray, R.A.\ Wilson and the second author, we show that the index of the soluble radical of a finite group $G$ is bounded in terms of the values of $\beta_G(u)$ for $2$-elements $u$.
\smallskip


\noindent AMS classification (2020): 20D10, 20D05, 20F45 
\end{abstract}

\section{Introduction}
In $2005$, Bray, Wilson and Wilson \cite{bray2005characterization} characterized the finite soluble groups by laws in two variables, by describing a recursively defined sequence $s_n(x, y)$ of words with the property that a finite group $G$ is soluble if and only if $s_n$ is a law in $G$ for all sufficiently large integers $n$.  The characterization can be reformulated as follows.  For each element $u$ of a group $G$ define $\theta_u\colon G\to G$ by $$\theta_u(g)=[g^{-u},g]$$ (where $[x,y]=x^{-1}y^{-1}xy$, $x^y = y^{-1}xy$ and $x^{-y} = (x^y)^{-1}$), and let $$\Theta_G(u)=\{g\in G\mid \theta_u^n(g)=g \hbox{ for some } n>0\}.$$  Since $\theta_u^n(g)$ lies in the $n$th term of the derived series of $G$ for all $u,g\in G$ and all positive integers $n$, if $G$ is soluble then $\Theta_G(u)=\{1\}$ for each $u$. The result in  \cite{bray2005characterization} shows the converse for finite groups: if $G$ is finite and $\Theta_G(u)=\{1\}$ for all $u$ then $G$ is soluble. 

In this paper, we are concerned with groups $G$ for which the sets $\Theta_G(u)$ are not necessarily trivial, but are rather small.
Clearly for each element $u$ the map $\theta_u$ acts as a permutation on the set $\Theta_G(u)$, and $\{1\}$ is an orbit; we call the other orbits {\em eventual orbits} of $\theta_u$ and write $\beta_G(u)$ for the number of such orbits in $G$.  Thus the result of 
 \cite{bray2005characterization} shows that the group $G$ is soluble if and only if $\beta_G(u)=0$ for each element $u\in G$.
We recall that the soluble radical $\rR(G)$ is the unique largest normal soluble subgroup of a finite group $G$.  It is reasonable to speculate that if $|\Theta_G(u)|$ is small for each $u\in G$ then the quotient $G/\rR(G)$ should also be small.  We prove more than this. 

\begin{theorem}
	Let $G$ be a finite  group and suppose that $\beta_G(u)\leqslant \gk$ for each $2$-element $u\in G$. Then
	$$|G\colon \rR(G)| \leqslant \gk^{200+\log \log \gk}.$$
\end{theorem}

We do not know whether there is a polynomial bound for $|G\colon \rR(G)|$ in terms of $\kappa$. 

The crucial ingredient in the proof of Theorem 1 is the following result:   

\begin{theorem}
Let $G$ be a finite almost simple group such that $\beta_G(u)\leqslant \gk$ for each $u\in G$ of order $2$ or $4$. Then $\lvert G \rvert \leqslant \gk^{200}$.
\end{theorem} 

Our proofs give a smaller power than $200$ in these theorems and the power could certainly be reduced yet further at the cost of more explicit calculation.  It will also become clear that there are much smaller asymptotic bounds for simple groups.   

On the way to establishing these results we obtain the following result:

\begin{theorem}\label{atleast8} \begin{enumerate} \item[\rm(a)] Each minimal simple group $S$ has an involution $u$ with $\beta_S(u) \geqslant 8$.
\item[\rm(b)] Each finite non-soluble group $G$ has a $2$-element $v$ with $\beta_G(v) \geqslant 8$. \end{enumerate}
	\end{theorem}

We note that the bound in this result is best possible.  It is easy to calculate by hand that for $u=(1\ 2)(3\ 4)\in {\rm A}_5$ the eventual $\theta_u$-orbits are four 
of length $2$ consisting of $3$-cycles, with 
representatives $(1\ 5\ 3)^{\pm1}, (2\ 4\ 5)^{\pm1}$, and four of length $4$ consisting of $5$-cycles, with representatives $(1\ 2\ 3\ 4\ 5)^{\pm1}$ and $(5\ 3\ 4\ 2\ 1)^{\pm1}$.

Similar investigations could be carried out with the maps $\theta_u$ replaced by other maps related to characterizations of group properties.  For  example, for $G$ finite and $u\in G$, define $\varepsilon_u\colon G\to G$ by $\varepsilon_u(x)=[x,u]$.  The powers of $\varepsilon_u$ are the Engel maps $x\mapsto [x,u,\dots,u]$, and Zorn's Theorem \cite{zorn} shows that $G$ is nilpotent if and only if no $\varepsilon_u$ has non-trivial eventual orbits. 
Let $q$ be a prime power and consider the split extension of the additive group $A$ of $\F_q$ by the multiplicative group $H$.  Thus $A$ is an $H$-module and for all $a\in A$, $h\in H$ the restriction to $A$ of 
$\varepsilon_{ah}$ is the map $x\mapsto x(h-1)$.  Therefore $\varepsilon_{ah}$ has $|H\colon \langle h-1\rangle|$ eventual orbits if $h\neq1$.  If $q-1$ is a Mersenne prime then this index is $1$.  Therefore deciding whether the number of eventual orbits of maps $\varepsilon_u$ bounds the index of the Fitting subgroup of a finite group has a very different character from the proof of Theorem 1 of this paper.

\section{Preliminary results and reductions}

Although it seems hard to work directly with the maps $\theta_u$ and their orbits, they behave well with respect to subgroups, quotients and direct products.  

\begin{lemma}\label{basic}  Let $G$ be a finite group.
 \begin{enumerate} \item[\rm(a)] If $H\leqslant G$ then $\beta_H(u)\leqslant \beta_G(u)$ for each $u\in H$.
\item[\rm(b)] If $K\triangleleft G$ then $\beta_{G/K}(uK) \leqslant \beta_G(u)$ for each $u\in G$. 
\item[\rm(c)] If $Y\leqslant\rZ(G)$ then $\beta_{G/Y}(uY) = \beta_G(u)$ for each $u\in G$. 
\item[\rm(d)] Let $G_1, \dots, G_n$ be finite groups and suppose for each $i$ that $G_i$ has a $2$-element $u_i$ with $\beta_{G_i}(u_i)= \gk_i$.
		Then $D = G_1 \times \dots \times G_n$ has a $2$-element $u$ with $\beta_D(u)+1= (\gk_1+1) \dots (\gk_n+1)$.  \end{enumerate}
	\end{lemma}
	\begin{proof}  Assertions (a)--(c) are clear.  To prove (d),
		it suffices by induction to establish the result for $n=2$. Let $x_1, \dots, x_{d_1}$ be in distinct $\theta_{u_1}$-orbits in $G_1$ (including the trivial orbit $\{1\}$) and $y_1, \dots, y_{d_2}$ in distinct $\theta_{u_2}$-orbits in $G_2$. Since $\theta^r_{u_1u_2}(x_iy_j)=\theta^r_{u_1}(x_i)\theta^r_{u_2}(y_j)$ for each $r>0$
		the elements $x_i y_j \in G_1 \times G_2$ are in distinct $\theta_{u_1u_2}$-orbits in $G_1 \times G_2$, and if $x_i,y_j$ are not both $1$ then the orbit of $x_iy_j$ is non-trivial.
	\end{proof}
	
The following lemma and the idea in its proof will recur throughout our treatment of the finite simple groups.  

\begin{lemma} \label{generalestimate}
 Let  $G$ be a group, $L$ a minimal non-soluble subgroup and $u \in L$.  Then $$\beta_G(u)\geqslant  \frac{\lvert \rC_G(u) \rvert}{\lvert \rC_G(u)\cap\rN_G(L) \rvert}\beta_L(u)\geqslant \frac{|\rC_G(u)|}
 {|\rC_G(L)|\,|\rC_{\Aut(L)}(u)|}\beta_L(u).$$
 \end{lemma}

\begin{proof}  Let $x$ be a non-trivial element of $\Theta_L(u)$. Then $\langle x, u \rangle = L$ and $\langle \theta^n_u(x), u \rangle = L$ for each $n\geqslant0$.  
For any $h \in C_G(u)$ observe that $(\theta_u(x))^h = \theta_u(x^h)$ and hence $(\theta_u^n(x))^h = \theta_u^n(x^h)$ for each $n \geqslant 0$ so the lengths of the eventual orbits containing $x$ and $x^h$ are the same. In particular, this means that $\rC_G(u)$ acts on $\Theta_G(u)$ by conjugation. If $x\in\Theta_L(u)$, $g \in \rC_G(u)$ and $L \neq L^g$, then $\langle x^g , u \rangle = L^g \neq L$ and so $x$ and $x^g$ do not belong to the same $\theta_u$-orbit since all elements in the same orbit generate the same conjugate of $L$. Thus each $L^g$ contains $\beta_L(u)$ orbits of $\theta_u$.  The stabilizer of $L$ in the conjugation action of $\rC_G(u)$ on the conjugates of $L$ is $\rN_{\rC_G(u)}(L)=\rC_G(u)\cap\rN_G(L)$, and $(\rC_G(u)\cap\rN_G(L))/\rC_G(L)$ is isomorphic to a subgroup of the centralizer of $u$ in $\Aut(L)$. Therefore by the orbit-stabilizer theorem,
$$\beta_G(u) \geqslant  \frac{\lvert \rC_G(u) \rvert}{\lvert \rC_G(u)\cap\rN_G(L) \rvert}\beta_L(u)  \geqslant \frac{\lvert \rC_G(u) \rvert}{\lvert \rC_G(L) \rvert \cdot \lvert \rC_{\Aut(L)}(u) \rvert} \beta_L(u)$$
as required. 
\end{proof}	

The above inequalities are in many cases sufficient to give a polynomial bound on the order of $\Aut(G)$ for a simple group $G$ in terms of the maximum value of $\beta_G(u)$, as we shall see in the next sections.  

The result below is easily verified using the computer software package GAP. 

\begin{lemma}\label{GAP} The conclusion of Theorem $\ref{atleast8}$ holds for the minimal simple groups $\pL_2(q)$ with $q \in \{7,8,13,17,23,27,32 \}$ and for $\pL_3(3)$ and $Sz(8)$.
\end{lemma}

In later sections we shall prove the following two results.

\begin{lemma} \label{PSL2}
		Let $q \geqslant 4$ be a prime power and $G = \SL_2(q)$ or $\pL_2(q)$. Assume that $\beta_G(u)\leqslant \gk$ for each $u \in G$ of order $2$ or $4$. Then $\lvert \Aut(G) \rvert \leqslant \gk^{7}$.
\end{lemma}

\begin{lemma}\label{suzuki}
		Let $m \in \N$ and $G =Sz(q)$ for $q=2^{2m+1}$.  If $\beta_G(u)\leqslant \kappa$ for each involution $u \in G$ then $\lvert \Aut(G) \rvert \leqslant \kappa^{12}$. 
	\end{lemma} 	

From the results stated above we can now prove Theorem $\ref{atleast8}$ and show that Theorems 2 and 3 imply Theorem 1.
	
\begin{proof}[Proof of Theorem $\ref{atleast8}$]
	Every non-soluble group $G$ contains a minimal simple group $S$ as a subquotient, and each $2$-element in $S$ is the image of a $2$-element in $G$.  Thus it suffices to establish assertion (a). By Thompson's classification \cite{thompson}, all minimal simple groups are found among the groups $\pL_2(q)$ for a prime power $q$, the Suzuki groups $Sz(2^p)$ for a prime $p$ and $\pL_3(3)$.  From Lemmas \ref{PSL2} and \ref{suzuki} we have $\beta_{\pL_2(q)}(u) \geqslant \frac{1}{4}(q-5)$ and $\beta_{Sz(2^p)}(u) \geqslant \frac{1}{2}(2^p-2)$ for suitable involutions $u$. We have already noted the result for ${\rm A}_5$, which is isomorphic to $L_2(4)$ and $L_2(5)$.  The remaining groups minimal simple groups are precisely those covered by the results reported in Lemma \ref{GAP}.
\end{proof}
	 
\begin{proof}[Deduction of Theorem $1$ from Theorems $2$ and $3$] Let $\gk>0$ and $G$ be a finite  group such that $\beta_G(u)\leqslant \gk$ for each $2$-element $u\in G$.   Since every $2$-element in $G/\rR(G)$ is an image of a $2$-element in $G$, this quotient inherits the hypothesis on $G$, and so it suffices to prove the result for $G/\rR(G)$.  Thus we may assume that $\rR(G)=1$; then the socle of $G$ is a direct product $S_1\times\cdots\times S_n$ of non-abelian finite simple groups, and $G$ permutes these groups by conjugation, with kernel $K$, say.   Write $\gk_i$ for the maximum value of $\beta_{S_i}(u_i)$ over all $2$-elements in $S_i$ for each $i$.  By Lemma \ref{basic} we have $\prod(\gk_i+1)\leqslant \gk+1$.  Since $K/\rC(S_i)$ is isomorphic to a subgroup of $\Aut(S_i)$ we have $|K/\rC(S_i)|\leqslant \gk_i^{200}$ for each $i$ from Theorem 2, and hence $|K|\leqslant \gk^{200}$ since $\bigcap\rC(S_i)=\rC(S_1\times\cdots\times S_n)=1$.  From Theorem \ref{atleast8} we have $\gk_i\geqslant8$ for each $i$ and so $9^n\leqslant \gk+1$. Hence $8^n \leqslant \gk$.
	 Since $G/K$ permutes the groups $S_i$ faithfully by conjugation we have $|G/K|\leqslant n!\leqslant n^{n/2}$ and so $$\log|G/K|\leqslant{\textstyle\frac12}\, n \log n \leqslant{\textstyle\frac12}(\log\gk)(\log\log\gk),$$ with logarithms to base $8$,
	 and $|G/K|\leqslant \gk^{(1/2) \log\log \gk}$.  Therefore $|G|\leqslant \gk^{200 +\log\log \gk}$, and the assertion of the theorem follows. \end{proof}

Therefore it remains to prove Theorem 2 and Lemmas \ref{PSL2} and \ref{suzuki}.  We use the classification of the finite simple groups, which asserts that every non-abelian finite simple group is isomorphic to one of the following:
\begin{enumerate}
	\item[] an alternating group ${\rm A}_n$ for $n \geqslant 5$;
	\item[] a simple group of Lie type in one of the families below:
	\begin{itemize}
		\item Classical Chevalley groups: $A_n(q)$, $B_n(q)$, $C_n(q)$, $D_n(q)$;
		\item Exceptional Chevalley groups: $E_6(q)$, $E_7(q)$, $E_8(q)$, $F_4(q)$, $G_2(q)$;
		\item Steinberg groups: $^2A_n(q^2)$, $^2D_n(q^2)$, $^2E_6(q^2)$, $^3D_4(q^3)$;
		\item Suzuki--Ree groups: $^2B_2(q)$, $^2F_4(q)$, $^2G_2(q)$;
	\end{itemize}
	\item[] one of the $27$ sporadic groups (including the Tits group).
\end{enumerate}
We handle the groups listed in the various classes above in the following sections of the paper.

\section{Alternating Groups}
\begin{lemma}\label{Alternating}
	Let $G = {\rm A}_n$ be the alternating group of degree $n\geqslant5$.   If $\beta_G(u)\leqslant \gk$ for each involution $u\in G$ then $\lvert \Aut(G) \rvert \leqslant \gk^4$.
\end{lemma}
	
\begin{proof}
First, as already noted, for $v=(1\ 2)(3\ 4)\in {\rm A}_5$ we have $\beta_{{\rm A}_5}(v)=8$.
For $6\leqslant n\leqslant9$, the element $v$ lies in the alternating group on $\{1,2,3,4,s\}$ for $s\in\{5,\dots, n\}$ and so since these are distinct minimal simple subgroups 
the idea in the proof of Lemma \ref{generalestimate} shows that $\beta_{{\rm A}_n}(v) \geqslant 8(n-4)$, and it is easy to check that $(8(n-4))^4 \geqslant 2(n!)$ for $n\in\{6,7,8,9\}$.  

Suppose that $n\geqslant10$. Write $n=5m+r$ with $m\geqslant2$ and $0\leqslant r\leqslant 4$. We choose an action of ${\rm A}_5$ on $\{1,\dots,n\}$ with orbits $\{5k+1,\dots,5k+5\}$ for $k=0,\dots,m-1$ and $\{s\}$ for $s>5m$, and with $v\in {\rm A}_5$ acting
as the involution $$u=\prod_{k=0}^{m-1} \big((5k+1) \ (5k+2)\big)(\big(5k+3) \ (5k+4)\big).$$  Let $L$ be the image of ${\rm A}_5$ in this action.
	
The conjugacy class of $u$ in ${\rm A}_n$ for $m \geqslant 2$ contains all elements of the same cycle type since $u$ is centralized by the odd permutation $(5 \ 10)$.
Hence
	$$ \lvert \rC_{{\rm A}_n}(u) \rvert = \frac{n!/2}{ \binom{n}{2} \binom{n-2}{2} \cdots \binom{n-4m+2}{2}/(2m)! }= (n-4m)! \ (2m)! \ 2^{2m-1}. $$

Since $\rC_{{\rm A}_n}(L)$ maps $L$-orbits to $L$-orbits, we
have a  homomorphism  $\rC_{{\rm A}_n}(L) \to {\rm S}_m$; since $\rC_{{\rm S}_5}({\rm A}_5)$ is trivial, the kernel fixes each $k\leqslant 5m$ and so has order dividing $r!$.  In particular, $\lvert \rC_{{\rm A}_n}(L) \rvert \leqslant (m!)(r!)$ holds.  Moreover, $\rC_{\Aut (L)}(u)\cong \rC_{{\rm S}_5}((1\ 2)(3\ 4))$ which has order $8$.  Thus, by Lemma 2, 
$$\beta_{{\rm A}_n}(u)\geqslant \frac{|\rC_{{\rm A}_n}(u)|}
 {|\rC_{{\rm A}_n}(L)|\,|\rC_{\Aut(L)}(u)|}\beta_L(u)\geqslant \frac{(n-4m)! \cdot(2m)! \cdot 2^{2m-1}}{m!\cdot r! \cdot 8}\cdot 8,
$$ and so $$\beta_{{\rm A}_n}(u) \geqslant \begin{pmatrix} m+r\cr r\end{pmatrix} (2m)! \cdot2^{2m-1}.\eqno(1)$$ 
If $m=2$ then (1) gives
$$\beta_{{\rm A}_n}(u)\geqslant \begin{pmatrix} 2+r\cr 2\end{pmatrix} 4!\cdot2^3$$ and direct calculation from this shows that $\beta_{{\rm A}_n}(u)^4\geqslant n!$ for $10\leqslant n\leqslant 14$.

If $m\geqslant3$, from (1) we have
$$3^{6m} \beta_{{\rm A}_n}(u)^3\geqslant \bigg(\prod_{k=0}^{2m-1} (6m-3k)(6m-3k-1)(6m-3k-2)\bigg) \, 2^{6m-3} \geqslant (6m)! \,2^{6m-3}$$
and so $$\frac{\beta_{{\rm A}_n}(u)^4}{n!}\geqslant \bigg(\frac{2^8}{3^6}\bigg)^m\frac1{16} \cdot (2m)!\cdot\frac{(6m)!}{n!}\geqslant \bigg(\frac13\bigg)^m\frac1{16} \cdot (2m)!\cdot\frac{(6m)!}{n!}$$
It is easy to check that the right-hand side above is greater than $1$ when $6m\geqslant n$ and to check from (1) that the left-hand side is greater than $1$ if $n=19$.
\end{proof}

	\section{Groups of type $\SL_n(q)$ and $\pL_n(q)$}
	The treatment of finite simple groups of Lie type is more involved.  In this section we begin by proving the Lemma \ref{PSL2} (stated in Section 2); then we use the information that it provides about the groups $\SL_2(q)$ and $\pL_2(q)$ to settle the case of groups $\SL_n(q)$ and $\pL_n(q)$ with $n\geqslant3$.  We require the following information about their automorphism groups: it follows directly from results due essentially to Dieudonn\'e \cite{dieudonne}.  For (b), see also p.\ xvi in the Atlas \cite{atlas}.  
	
	\begin{lemma}\label{boundAut}
		Let $q = p^f$ be a prime power and suppose that either $n\geqslant3$ or $n= 2$ and $q\geqslant4$.   Then
		\begin{enumerate} 
		\item[\rm(a)]  The automorphism groups of $\SL_n(q)$ and $\pL_n(q)$ are isomorphic and have order at most $q^{n^2}$. 
		\item[\rm(b)] The outer automorphism group $\Out (\pL_n(q))$ of $\pL_n(q)$ satisfies
		$$\lvert \Out (\pL_n(q)) \rvert =  \begin{cases} \gcd(q-1,n) \cdot f & \hbox{for  } n=2 \\ \gcd(q-1,n) \cdot 2f \quad&  \hbox{for  }n \geqslant 3. \end{cases}$$
		 \end{enumerate}
	\end{lemma}
		
	%\begin{lemma}
	%	Let $q \geqslant 4$ be a prime power and $G = \SL_2(q)$ or $\pL_2(q)$. Assume that $\beta_G(u)\leqslant \gk$ for each $u \in G$ of order $2$ or $4$. Then $\lvert \Aut(G) \rvert \leqslant \gk^{7}$.
	%\end{lemma}
	
	\begin{proof}[Proof of Lemma $\ref{PSL2}$]   By Lemmas \ref{basic}(c) and \ref{boundAut} it will suffice to prove the result for $G=\SL_2(q)$.
		We draw heavily on the ideas and the notation in \cite{bray2005characterization}. In particular we choose
		$$ u = \begin{pmatrix}
		0 & 1 \\
		-1 & 0
		\end{pmatrix}, $$ and for  $$
		w = \begin{pmatrix}
		a & b \\ c & d
		\end{pmatrix} \in \SL_2(q) \quad \text{ we write } \quad \theta_u(w) = \begin{pmatrix}
		A & B \\ C & D
		\end{pmatrix}.$$
		There are two cases depending on the characteristic of the underlying field.
		
\medskip	\noindent {\em Case} 1.   Suppose that $\char\, \F_q = 2$. 

We claim that $|\Aut(G)|\leqslant\kappa^6$.   This holds for $q=4$ by Lemma \ref{Alternating} since
$\SL_2(4)\cong {\rm A}_5$.  By Lemma \ref{GAP}, if $q=8$ we have $\gk^4 \geqslant 8^4 \geqslant |\Aut(\SL_2(8))|$.  Suppose now that $q\geqslant16$.  
		
		Theorem 2.1 in \cite{bray2005characterization} shows that $0 \not\in \{ a+d, b+c, a+b+c+d \}$ if and only if $0 \not\in \{ A+D, B+C, A+B+C+D \}$. In particular, this is the case for any element
		$$w = \begin{pmatrix}
		1 & 1 \\
		\mu & \mu+1
		\end{pmatrix}$$
		with $\mu \in \F_q \setminus \F_2$.
		Consider the quantity $y=\frac{a+d}{b+c}$ and let $Y= \frac{A+D}{B+C}$. An easy calculation shows that $Y = y^{-1}$, so when we apply powers of the map $\theta_u$
		the quantity $y$ has cycles of length $2$. For the elements $w$ above, it takes the value $\frac{\mu}{\mu+1}=1 - \frac{1}{\mu+1}$, which can be any element of $\F_q \setminus \F_2$. So $\beta_G(u) \geqslant \frac12(q-2)$.
		It is easy to check that $(\frac12(q-2))^6\geqslant q^4$ for $q\geqslant16$, and the claim follows from Lemma \ref{boundAut}.

\medskip \noindent{\em Case} 2.  Suppose that $\char\, \F_q = p \geqslant 3$.
		
		Let $Q = \{\mu^2 \ | \ \mu \in \F_q^\times \}$ and $N = \F_q^\times \setminus Q$.  Theorem 2.2 in \cite{bray2005characterization} shows that any element satisfying $a-d \neq 0$, $b+c \neq 0$ and $-2((a-d)^2+(b+c)^2) \in N$ leads to a (non-trivial) eventual orbit of $\theta_u$. In particular, we consider
		$$w = \begin{pmatrix}
		\mu & 0 \\
		\lambda (\mu - \mu^{-1}) & \mu^{-1}
		\end{pmatrix}$$
		with $\mu^2 \notin \{0, 1 \}$ and $\lambda \neq 0$. The outstanding condition for $\lambda$ is now $-2(\lambda^2+1) \in N$ and we shall now show that there are many suitable choices for $\lambda$. If $-2$ is a square root in $\F_q$, we are interested in the number of non-zero $\lambda$ such that $\lambda^2+1 \in Q$. If $-2$ is \emph{not} a square root in $\F_q$, we are interested in the number of non-zero $\lambda$ such that $\lambda^2+1 \in N$. We shall now count the number of suitable $\lambda$ in both of these two cases.
		
		First suppose that $\lambda^2+1 \in Q$, so $\lambda^2+1=\nu^2$ for some $\nu \in \F_q^\times$; then $(\lambda - \nu)(\lambda + \nu)=-1$.  Let $t := \lambda - \nu \in \F_q^{\times}$; then $\lambda+\nu = -t^{-1}$. Therefore $\lambda = \lambda(t) = \frac12(t-t^{-1})$. If $\lambda(t_1) = \lambda(t_2)$, then $t_2=t_1$ or $t_2 = -t_1^{-1}$ and $\lambda(t)=0$ only for $t= \pm 1$. So in total there are $q-3$ values for $t$ such that $\lambda \neq 0$ and $\lambda^2+1\in Q\cap\{0\}$. Therefore the number of non-zero values of $\lambda$ for which $\lambda^2+1\in Q$ is $\frac12(q-5)$ if $-1\in Q$ and $\frac12(q-3)$ otherwise.  This also means that the number of non-zero values of $\lambda$ with $\lambda^2+1\in N$ is $\frac12(q+3)$ if $-1\in Q$ and $\frac12(q+1)$ otherwise.  Therefore, in all cases we have at least $\frac12(q-5)$ suitable choices for $\lambda$.
		
		Finally, consider the quantity $y=\frac{a-d}{b+c}$ and let $Y= \frac{A-D}{B+C}$. Following a calculation in \cite{bray2005characterization} we have $Y = -y^{-1}$, so the quantity $y$ has cycles of length $2$. For the elements $w$ and $\theta_u^{2n}(w)$ with $n>0$ it takes the value $\lambda^{-1}$ and so $w$ leads to an eventual orbit on which its values are $\lambda^{-1}$ and $-\lambda$. Therefore $\kappa\geqslant \beta_G(u)\geqslant \frac14(q-5)$. 
		For $q\geqslant 37$, we have $( \frac14(q-5))^7\geqslant q^4$ and the conclusion follows from Lemma \ref{boundAut}.  For $q \leqslant 31$,  Lemma \ref{GAP} gives $\kappa^{7} \geqslant 8^{7} =2^{21}> 32^4 > q^4 \geqslant \lvert \Aut(G) \rvert$.
	\end{proof}

	Next we consider all groups  $\SL_n(q)$ and $\pL_n(q)$. The following lemma simplifies the calculation of centralizer orders. 
	\begin{lemma}\label{centralizerofL}  
	\begin{enumerate} \item[\rm(a)] Let $L$ be a non-abelian group, $F$ a finite field and $M$ a simple $FL$-module of prime degree $\ell$ over $F$ on which $L$ acts faithfully. Then $\End_{FL}(M)$ consists of scalar multiplications by elements of $F$.
	
	\item[\rm(b)] Suppose in addition that $L \leqslant G$ and $N$ is an $FG$-module on which $G$ acts faithfully. If $N$ is isomorphic to $M^m$ regarded as an $FL$-module then $\lvert \rC_{G}(L) \rvert \leqslant |\GL_m(F)|$. 
	\end{enumerate} \end{lemma}
	\begin{proof}
		(a) The subring $E=\End_{FL}(M)$ of $\End_{F}M$ is a field by Schur's Lemma and it contains the field $F_1$ of scalar multiplications by elements of $F$.
		Moreover $M$ is an $EL$-module, with $\dim_EM=(\dim_{F_1}M)/[E\colon F_1]$.  Since $L$ is non-abelian we have $\dim_EM>1$, and the result follows.
	
		(b) By (a) we have $$\Hom_{FL}(N,N) = \Hom_{FL}(M^m,M^m)\cong \mathrm{Mat}_m(F).$$ 
	Since $G$ embeds in the group of units of the algebra on the left-hand side the result follows. 
			\end{proof}

	\begin{lemma} \label{mainlemma}
		Let $q$ be a prime power, $n \geqslant 3$ and $G = \SL_n(q)$ or $\pL_n(q)$. Assume that $\beta_G(u)\leqslant \kappa$ for each $2$-element $u \in G$. Then $\lvert \Aut(G) \rvert \leqslant \kappa^{22}$. 
	\end{lemma}
	\begin{proof} By Lemmas \ref{basic}(c) and \ref{boundAut} it suffices to prove the result for the groups $G = \SL_n(q)$.
		We distinguish between three cases depending on $q$. \medskip
		
		\noindent{\em Case} 1.  Suppose that $q=2$.
		
		Assume first that $n=3m$ for some $m \in \N$ and let
		$$u={\rm diag}\left(\begin{pmatrix}
		1 & 0 & 0\\
		0 & 1 & 1\\
		0 & 0 & 1
		\end{pmatrix}, \dots, 
		\begin{pmatrix}
		1 & 0 & 0\\
		0 & 1 & 1\\
		0 & 0 & 1
		\end{pmatrix} \right)$$
		and $L = \left\lbrace \text{diag}(a, \dots, a) \mid a \in \SL_3(2) \right\rbrace$; thus $L\cong \SL_3(2)$. Observe that $u$ is similar to the block matrix
		$$M=\begin{pmatrix}
		I_m & 0 & 0\\
		0 & I_m & I_m\\
		0 & 0 & I_m
		\end{pmatrix}$$
		where $I_m$ and $0$ denote the $m \times m$ identity and zero matrices. We have
		$$\rC_{\SL_{3m}(2)}(M)=\left\lbrace \begin{pmatrix}
		A & 0 & C\\
		D & B & E\\
		0 & 0 & B
		\end{pmatrix} \mid A, B \in \SL_m(2), C, D, E \in {\rm Mat}_m(2) \right\rbrace$$
		and it follows that $\lvert \rC_{\SL_{3m}(2)}(u) \rvert = \lvert \rC_{\SL_{3m}(2)}(M) \rvert = 2^{3m^2} \lvert \SL_{m}(2) \lvert ^2$.
		We also have $|\SL_m(2)|=\prod_{k=0}^{m-1}(2^m-2^k)\geqslant 2^{(m-1)m}$. 
				
		Since the natural module for $\SL_3(2)$ is simple, by Lemma \ref{centralizerofL} we have $\lvert \rC_{\SL_{3m}(2)}(L) \rvert \leqslant |\GL_m(2)|$.
		From Lemma \ref{boundAut} we have $|\Out(\SL_3(2))|=2$ and so $\lvert \Aut(\SL_3(2)) \rvert=336$. By Lemma \ref{generalestimate}, we obtain the estimate
		$$ \beta_G(u) \geqslant \frac{\lvert \rC_{\SL_{3m}(2)}(u)\rvert}{\lvert \rC_{\SL_{3m}(2)}(L) \rvert \cdot \lvert \Aut(\SL_{3}(2)) \rvert} \cdot 8 \geqslant\frac{2^{3m^2} \lvert \SL_m(2) \rvert}{42}  \geqslant 2^{4m^2-m-6}.$$
		Now for $G = \SL_n(2)$ where $n = 3m + r$ with $r \in \{0, 1, 2\}$, we can embed $\SL_{3m}(2)$ in $\SL_n(2)$. Then $\lvert \Aut(G) \rvert \leqslant 2^{n^2} \leqslant  2^{9m^2+12m+4}$.  Since we have $8(4m^2-m-6)\geqslant 9m^2+12m+4$ for $m\geqslant2$ it follows that $\gk^8\geqslant |\Aut(G)|$ if $m\geqslant2$.  
		 For $m=1$ we have $\beta_G(u) \geqslant 8$ by Lemma \ref{GAP} because $\SL_3(2) \cong \pL_2(7)$. Therefore $\gk^9 \geqslant 8^9 =2^{27} \geqslant \lvert \Aut(G) \rvert$.
		 		
		\medskip \noindent {\em Case} 2.  Suppose that $q=3$.
		
		Assume first that $n=3m$ for some $m \in \N$ and let
		$$u = {\rm diag}\left( \begin{pmatrix}
		1 & 0 & 0 \\ 0 & 2 & 0 \\ 0 & 0 & 2
		\end{pmatrix}, \dots, \begin{pmatrix}
		1 & 0 & 0 \\ 0 & 2 & 0 \\ 0 & 0 & 2
		\end{pmatrix} \right)$$
		and $L = \left\lbrace \text{diag}(a, \dots, a) \mid a \in \SL_3(3) \right\rbrace$; thus $L \cong \SL_3(3)$.
		
		Then $\rC_{\GL_{3m}(3)}(u) \cong \GL_{2m}(3) \times \GL_{m}(3)$. Intersecting with $\SL_{3m}(3)$ we obtain that $\lvert \rC_{\SL_{3m}(3)}(u) \rvert = \frac{1}{2}\lvert\GL_{2m}(3)\rvert\cdot\lvert\GL_{m}(3)\rvert$. 
		By Lemma \ref{centralizerofL} we have $\lvert \rC_{\SL_{3m}(3)}(L) \rvert \leqslant |\GL_m(3)|$.
				Now $|\pL_3(3)|= 2^4\cdot3^3\cdot13$ and so by Lemma 4 we have $\lvert \Aut(\SL_{3}(3)) \rvert = 2^5\cdot3^3\cdot13$. Therefore by Lemma \ref{generalestimate} we obtain 
		$$ \beta_G(u) \geqslant \frac{\frac12|\GL_{2m}(3)|\cdot |\GL_{m}(3)| }{\lvert \GL_{m}(3)\rvert \cdot (2^5\cdot3^3\cdot13) } \cdot 8 \geqslant \frac{\prod_{k=1}^{2m}(3^{2m}-3^{2m-1})}{2^3\cdot3^3\cdot13} = \frac1{13}\cdot 3^{2m(2m-1)-3}\ 2^{2m-2}.$$
		Since $26<3^3<2^5$ we obtain $\log_3\beta_G(u)\geqslant 2m(m-1)-6+\frac35(2m-2)$.
		
		For a general $G = \SL_n(3)$ where $n = 3m + r$ for $r \in \{0, 1, 2\}$, we can embed $\SL_{3m}(3)$ in $\SL_n(3)$. Then $\lvert \Aut(G) \rvert \leqslant 3^{n^2} \leqslant  3^{9m^2+12m+4}$.  Since $7(2m(m-1)-6+\frac35(2m-2))\geqslant 9m^2+12m+4$ for $m\geqslant2$ we conclude that $\gk^7 \geqslant |\Aut(G)|$ for $m\geqslant2$.  
		  
		For $m = 1$ we have $\beta_G(u) \geqslant 8$  by Lemma \ref{GAP} and so $\kappa^{14} \geqslant 8^{14} =2^{42}>(2^8)^5> (3^5)^5 \geqslant \lvert \Aut(G) \rvert$.

		\medskip\noindent {\em Case} 3.  Suppose that  $q=p^f$ for $p\geqslant5$ or $p\in\{2,3\}$ and $f \geqslant 2$.
		
		Assume first that $n=2m$ for some $m \in \N$ and let
		$$u = {\rm diag}\left( \begin{pmatrix}
		0 & 1 \\ -1 & 0
		\end{pmatrix}, \dots, \begin{pmatrix}
		0 & 1 \\ -1 & 0
		\end{pmatrix} \right).$$		
		We split the calculation of $|\rC_G(u)|$ into a number of cases:			
		\begin{enumerate}
			
			\item [(i)] $p=2$: $\begin{pmatrix} 0 & 1 \\ 1 & 0 \end{pmatrix}$ is similar to $\begin{pmatrix} 1 & 1 \\ 0 & 1 \end{pmatrix}$ so $u$ is similar to $M=\begin{pmatrix}I_m & I_m \\ 0 & I_m \end{pmatrix}$ over $\F_q$. Since 
			$$\rC_{\GL_{2m}(q)}(M) = \left\lbrace \begin{pmatrix} A & B \\ 0 & A\end{pmatrix} \mid A \in \GL_{m}(q), \ B \in \text{Mat}_m(q) \right\rbrace$$
			we have $\lvert \rC_{\GL_{2m}(q)}(u) \rvert = q^{m^2}\lvert \GL_{m}(q) \rvert$.
			
			\item [(ii)] $p \geqslant 3$ and $\F_q$ has an element $i$ with $i^2=-1$: then $u$ is diagonalizable over $\F_q$ with eigenvalues $i$ and $-i$ each of multiplicity $m$ and so $\rC_{\GL_{2m}(q)}(u) \cong \GL_{m}(q) \times \GL_{m}(q)$.
			
			\item [(iii)] $p \geqslant 3$ and $\F_q$ has no element $i$ with $i^2=-1$:
			then $\rC_{\GL_{2m}(q)}(u) \cong \GL_{m}(q^2)$.
		\end{enumerate}
		The intersection of the centralizer with $\SL_{2m}(q)$ has index at most $q-1$, and so in each case above we have 
		$$\lvert \rC_{\SL_{2m}(q)}(u) \rvert \geqslant \frac{1}{q-1} \lvert \GL_{m}(q) \rvert^2 \geqslant  |\GL_m(q)| (q-1)^{m-1}q^{m^2-m}.$$

Next we define a subgroup $L$ containing $u$.  Again there are various cases: in each, we choose $L$ to be either minimal simple or a double cover of a minimal simple group and hence minimal non-soluble. 
		\begin{enumerate}
			\item [(i)] $p=2$.
			
			Let $L = \left\lbrace \text{diag}(a, \dots, a) \mid a \in \SL_2(2^r) \right\rbrace \cong \SL_2(2^r)$ where $r$ is some prime divisor of $f$.
			
			\item [(ii)] $p=3$ and $f$ has an odd prime divisor $r$.
			
			Let $L = \left\lbrace \text{diag}(a, \dots, a) \mid a \in \SL_2(3^r) \right\rbrace \cong \SL_2(3^r)$.
			
			\item [(iii)] $p=3$ and $f$ is a power of $2$.
			
			Consider the subgroup $H \leqslant \left\lbrace \text{diag}(a, \dots, a) \mid a \in \SL_2(9) \right\rbrace \cong \SL_2(9) \cong 2 \cdot A_6$ and let $L \leqslant H$ be any subgroup of $H$ isomorphic to $2 \cdot A_5$ such that $u \in L$. Such subgroups exist since there is a single conjugacy class of elements of order $4$ in $\SL_2(9)$. 
			
			\item [(iv)] $p \geqslant 5$ with $p \equiv \pm 2 \pmod 5$ or $p=5$.
			
			Let $L = \left\lbrace \text{diag}(a, \dots, a) \mid a \in \SL_2(p) \right\rbrace \cong \SL_2(p)$.

			\item [(v)] $p \geqslant 5$ with $p \equiv \pm 1 \pmod 5$.
			
			The subgroup $H = \left\lbrace \text{diag}(a, \dots, a) \mid a \in \SL_2(p) \right\rbrace \cong \SL_2(p)$ has a single conjugacy class of elements 
			of order $4$. Let $L \leqslant H$ be any subgroup isomorphic to $2 \cdot A_5$ such that $u \in L$.
		\end{enumerate}
		Lemma \ref{centralizerofL} shows that in each of these cases $\lvert \rC_{\GL_{2m}(q)}(L) \rvert \leqslant |\GL_{m}(q)|$.
		
		By Lemma \ref{generalestimate} we obtain
		$$ \kappa \geqslant \frac{\frac{1}{q-1} |\GL_{m}(q)|^2}{|\GL_{m}(q)| \cdot \lvert \Aut(L) \rvert} \geqslant \frac{(q-1)^{m-1}q^{m^2-m}}{q^4}= (q-1)^{m-1} q^{m^2-m-4}.$$
		
		For a general $G = \SL_n(q)$ where $n = 2m + r$ with $r \in \{0, 1\}$ we can embed $\SL_{2m}(q)$ in $\SL_n(q)$. Then $\lvert \Aut(G) \rvert \leqslant q^{n^2} \leqslant  q^{4m^2+4m+1}$ and so for any $m \geqslant 3$ we obtain
		\begin{align*}
		15 \log_q \gk &\geqslant 15 (m^2-m-4) + 15(m-1) \log_q(q-1)\\
		&\geqslant 15m^2-15m-60 + 15 (m-1) \log_43  \\
		&\geqslant 4m^2+4m+1 \geqslant \log_q |\Aut(G)|
		\end{align*}
		using the estimate $\log_43 \geqslant \frac34$. Hence $\gk^{15} \geqslant |\Aut(G)|$ for any $m \geqslant 3$. If $m = 2$ then $\SL_2(q) \times \SL_2(q)$ is a subgroup;  therefore $\kappa \geqslant \beta_{\SL_2(q)}(u)^2 \geqslant q^{8/7}$ by Lemma \ref{PSL2} and hence $\kappa^{22} \geqslant q^{25} \geqslant \lvert \Aut(G) \rvert$. Similarly, if $m=1$ then $\kappa^{18} \geqslant q^9 \geqslant \lvert \Aut(G) \rvert$.
		\end{proof}
	
	\section{The remaining groups of Lie type}
	In this section we shall establish inequalities of the necessary form for general groups of Lie type by embedding groups of type $\SL_n(q)$ or $\pL_n(q)$ in them as subgroups of small index and using the results of the previous section.
This strategy breaks down for the Suzuki groups $Sz(q)$, because they do not contain copies of $\SL_2(q)$ or $\pL_2(q)$.  They are dealt with using an independent proof at the end of the section.

First we note some properties of the groups of Lie type; for details we refer the reader to Carter \cite{carter1989simple}.

Let $\mathfrak{g}$ be a simple Lie algebra over $\C$ and let $\mathfrak{h}$ be a Cartan subalgebra. Let $\Phi$ be the root system associated to $\mathfrak{h}$ and choose a set  $\Delta$ of simple roots.

Let $G$ be a group of Lie type $\mathfrak{g}$ over a field $F$ and let $X_\alpha$ denote the root subgroup associated to the root $\alpha$. Then $L_\alpha := \langle X_\alpha, X_{-\alpha} \rangle$ is isomorphic to $\SL_2(F)$ or $\pL_2(F)$; the isomorphism arises from a surjective group homomorphism $\phi_\alpha: \SL_2(F) \to \langle X_\alpha , X_{-\alpha} \rangle$ under which
$$\begin{pmatrix}
1 & t \\ 0 & 1
\end{pmatrix} \mapsto x_\alpha(t) \quad \text{ and } \quad \begin{pmatrix}
1 & 0 \\ t & 1
\end{pmatrix} \mapsto x_{-\alpha}(t).$$
	The subgroups $L$ defined in the proof of Lemma \ref{mainlemma} for $q \geqslant 4$ were subgroups of the diagonal subgroup of a direct product of copies of $\SL_2(q)$, generated by groups $\langle X_\alpha, X_{-\alpha}\rangle$ for simple roots that are pairwise non-adjacent in the Dynkin diagram of $A_n$.
	
	We will use the following estimates for the orders of finite groups of Lie type:
	\begin{itemize}
		\item $|X_n(q)| \leqslant q^{2n^2+n}$ for classical groups of types $X \in \{ B, C, D\}$;
		\item $|{}^2D_n(q^2)| \leqslant q^{2n^2-n}$ and $|{}^2A_n(q^2)| \leqslant q^{n^2+2n}$;
		\item $|{}^2E_6(q^2)| \leqslant |E_6(q)| \leqslant q^{78}$, $|E_7(q)| \leqslant q^{133}$ and $|E_8(q)| \leqslant q^{248}$.
	\end{itemize}
	The reader is referred to the Atlas \cite[p.\ xvi]{atlas} for the exact orders.
	
	We will also need estimates for the orders of the outer automorphism groups of groups of Lie type over finite fields. The exact orders can be found in the Atlas \cite[p.\ xvi]{atlas}. For our purposes it will suffice that $|\Out(G)| \leqslant q^3$ for any Chevalley group $G=X_n(q)$ or twisted group $G={}^mX_n(q^m)$.
	\begin{lemma}
		Let $G$ be a finite classical group, i.e. a simple group of Lie type $A_n$, $B_n$, $C_n$, $D_n$, $^2A_n$ or $^2D_n$ such that $\beta_G(u)\leqslant \gk$ for each $2$-element $u \in G$. Then $\lvert \Aut(G) \rvert \leqslant \gk^{154}$.
	\end{lemma}
	\begin{proof}
		We have already established a stronger result in Lemma \ref{mainlemma} for groups $A_n(q)$. Let $G = X_n(q)$ be a Chevalley group of type $X$ where $X \in \{ B, C, D\}$ over $\F_q$. The Dynkin diagram of $A_{n-1}$ embeds as a subgraph in the Dynkin diagram of $X_n$. This inclusion corresponds to a root subsystem $\Phi'$ of type $A_{n-1}$ and it leads to a subgroup $H = \langle X_\alpha \ | \ \alpha \in \Phi' \rangle$ with $\pL_n(q)$ as a quotient; therefore $\gk^{22} \geqslant q^{n^2}$ by Lemma \ref{mainlemma}. We can further estimate $\lvert \Aut(G) \rvert \leqslant |G| \cdot q^3 \leqslant q^{2n^2+n+3}$ and hence for any $n \geqslant 2$ we have 
		$$\gk^{88} \geqslant q^{4n^2} \geqslant q^{2n^2+n+3} \geqslant \lvert \Aut(G) \rvert.$$
		
		Now let $G = {^2X_n(q^2)}$ be a twisted Chevalley group of type $X$ over $\F_{q^2}$ where $X \in \{A, D\}$. Recall that every element of $G$ is fixed by the Steinberg automorphism $gf$ where $g$ is a graph automorphism corresponding to a non-trivial symmetry of the Dynkin diagram $X_n$ and $f$ corresponds to the involutory field automorphism $x \mapsto x^q$. The orbits of $\langle g \rangle$ acting on $\Delta$ correspond to roots in the twisted root system of $G$, which is not necessarily reduced; see \cite[Section 2.3]{classification} for details about twisted root systems arising in the different cases.

		The twisted root system of ${}^2D_n$ is $B_{n-1}$ and so $G$ has  $L_{n-2}(q^2)$ as a subquotient. By Lemma \ref{mainlemma} we have $\gk^{22} \geqslant q^{(n-1)^2}$ and hence $$\gk^{88} \geqslant q^{4(n-1)^2} \geqslant q^{2n^2-n+3} \geqslant \lvert \Aut(G) \rvert$$ for any $n \geqslant 4$. For any $n \geqslant 2$ the twisted root system of ${}^2A_{2n-1}$ is $C_n$ and the twisted root system of ${}^2A_{2n}$ is $BC_n$. In each case $A_{n-1}$ embeds as a subgraph and hence $G$ contains $\pL_n(q)$ as a subquotient. By Lemma \ref{mainlemma} we have $\gk^{22} \geqslant q^{n^2}$ and hence $$ \gk^{154} \geqslant q^{7n^2} \geqslant q^{4n^2+4n+3} \geqslant |\Aut(G)|.$$ Finally, the result also holds for ${}^2A_2(q^2) = \text{PSU}_3(q)$ since $\pL_2(q) \cong \text{PSU}_2(q) \leqslant \text{PSU}_3(q)$ and $| \Aut(\text{PSU}_3(q)) | = |\text{PSU}_3(q)| \cdot q^3 \leqslant q^{11} \leqslant \gk^{20}$ by Lemma 5.
	\end{proof}
	
		\begin{lemma}
		Let $G$ be an exceptional Chevalley group $E_6(q)$, $E_7(q)$, $E_8(q)$, $F_4(q)$ or $G_2(q)$, an exceptional Steinberg group  $^2E_6(q^2)$ or $^3D_4(q^3)$ or a Ree group $^2F_4(q)$ or $^2G_2(q)$. Assume that $\beta_G(u)\leqslant \gk$ for each $2$-element $u \in G$. Then $\lvert \Aut(G) \rvert \leqslant \gk^{142}$.
	\end{lemma}
	\begin{proof}
		By considering Dynkin diagrams we obtain that $E_n(q)$ has $\pL_n(q)$ as a subquotient for $n \in \{ 6,7,8 \}$.  By Lemma \ref{mainlemma} we have $\gk^{22} \geqslant q^{n^2}$ and hence $\gk^{88} \geqslant q^{4n^2} \geqslant |\Aut(E_n(q))|$ for these values of $n$.
		
		For any other group $G$ in consideration, it suffices to note that $\SL_2(q)$ or $\pL_2(q)$ is a subgroup of $G$, so that $\gk^7 \geqslant q^4$ by Lemma 5. Hence $\gk^{142} \geqslant q^{81} \geqslant |{}^2E_6(q^2)| \cdot q^3 \geqslant |G| \cdot q^3 \geqslant |\Aut(G)|$.
	\end{proof}

Now we prove a corresponding result for the Suzuki groups, namely, the one stated above in Section 2.		

%\begin{lemma}\label{suzuki}
%		Let $m \in \N$ and $G =Sz(q)$ for $q=2^{2m+1}$.  If $\beta_G(u)\leqslant \kappa$ for each involution $u \in G$ then $\lvert \Aut(G) \rvert \leqslant \kappa^{12}$. 
%	\end{lemma} 
	\begin{proof}[Proof of Lemma $\ref{suzuki}$]
	For all of our notation throughout the proof, we refer the reader to \cite{bray2005characterization}.  As in that paper, take
	$$
u=T(0,1)=\begin{pmatrix}
1 & 0 & 0 & 0 \cr
0 & 1 & 0 & 0 \cr
1 & 0 & 1 & 0 \cr
1 & 1 & 0 & 1
\end{pmatrix}.
$$
Then $H:=\rC_G(u)=\{T(a,b)\mid a,b\in \F_q\}$ has order $q^2$ and it is normalized by the cyclic group $D=\{D(k)\mid k\in \F_q^\times \}$.  Thus $L=HD$ is a soluble subgroup and since $u\in L$ we have $\Theta_G(u)\setminus\{1\}\subseteq G\setminus L$.      Each element of $G\setminus L$ can be written uniquely
in the form $h_{1}dzh_{2}$ with $h_{1}$, $h_{2}\in H$, $d\in D$ and where
$$z:=\begin{pmatrix}
0 & 0 & 0 & 1 \cr
0 & 0 & 1 & 0 \cr
0 & 1 & 0 & 0 \cr
1 & 0 & 0 & 0
\end{pmatrix}.
$$
We calculate some entries of a typical element of $G\setminus L$: for $a, b, c, d \in \F_q$ and $k \in \F_q^\times$ we have
		$$E:=T(a, b)D(k)zT(c,d) =
		\begin{pmatrix}
		\;\;\cdot & k^{s/2+1}d & k^{s/2+1}c & k^{s/2+1} \\
		\;\;\cdot & \cdot & \cdot & k^{s/2+1}a \\
		\;\;\cdot & \cdot & \cdot & k^{s/2+1}(a^{1+s}+b) \\
		\;\;\cdot & \cdot & \cdot & \cdot \\
		\end{pmatrix} \eqno(2)$$
		where matrix entries replaced by dots are not important. 
				
Now let $x \in G \setminus L$ be an element such that the sequence $\theta_u^n(x)$ leads to a (non-trivial) eventual orbit. Then $x$ is conjugate (under an element of $H$) to an element in $HDz$ of the form $T(a, b)D(k)z$. The coefficients of the matrix $$M:=\theta_u(T(a, b)D(k)z)$$ are calculated in \cite[Lemma 3.2]{bray2005characterization}.
		
		Since $x$ leads to a (non-trivial) eventual orbit, we have $\theta_u(x) \in G \setminus L$ and this element is conjugate (under the element $T(e,f) \in H$ for some $e, f \in \F_q$) to a unique element of the form $T(A, B)D(K)z$. We now calculate $A, B, K$ in terms of $a, b, k$.
		\begin{align*}
		\theta_u(T(a, b)D(k)z) &= T(e, f)^{-1} T(A,B)D(K)zT(e, f) \\
		&= T(e, e^{s+1}+f)T(A,B)D(K)zT(e, f) \\
		&= T(e+A, eA^s+e^{s+1}+f+B)D(K)zT(e, f).
		\end{align*}
		The matrix on the right-hand side has the form in $(2)$ (after the substitution $a \mapsto e+A$, $b \mapsto eA^s + e^{s+1}+f+B$, $c \mapsto e$, $d \mapsto f$, $k \mapsto K$). Therefore
		\begin{align*}
		A&= (e+A)-e= (E_{24}-E_{13}) / K^{s/2+1} = (M_{24}-M_{13}) / K^{s/2+1}\\
		&= a(k^{2s+3}a^{s+1}+k^{2s+4} \xi) / K^{s/2+1}.
		\end{align*}
		In particular, if $a=0$, then $A=0$ too. We choose $a=0$ to facilitate the remaining calculations. 
		
		In a similar way we obtain
		$$K^{s/2+1}= E_{14} = M_{14} = k^{2s+4}(b+1)^s $$
		and after a few lines of calculation
		$$K = k^4(b+1)^{2s-2}.$$
		Finally $E_{12}=K^{s/2+1}f$ and $E_{34}=K^{s/2+1}(e^{s+1}+e^{s+1}+f+B)$ and so 
		$$K^{s/2+1}B =E_{34}-E_{12}=M_{34}-M_{12}= k^{2s+4}(b+1)^{s+1}.$$
		Hence $B=b+1$.  Therefore in the eventual orbit of $\theta_u$ to which $x$ leads,  the values $b,b+1$ alternate (as $\char \F_q=2$).
			
In \cite{bray2005characterization} it is shown that $x=T(0,b)D(1)z$ leads to a (non-trivial) eventual orbit for any $b \notin \F_2$. Therefore the number of eventual orbits of $\theta_u$ is at least the number of distinct pairs $\{b, b+1\}$ with $b \notin \F_2$, and so at least $\frac12(q-2)$. The group $\Out(G)$ is cyclic of order $2m+1 \leqslant q$ and so $\lvert \Aut(G) \rvert \leqslant q \lvert G \rvert \leqslant q^6$.  Since $q^6\leqslant (\frac{1}{2}(q-2))^{12}$ for all relevant $q$, the result follows.
	\end{proof}
	
	\section{Sporadic groups}

\begin{lemma}
Let $G$ be either a sporadic group or the Tits group ${}^2F_4(2)$.  Then $G$ has an involution $u$ with
$|\Aut(G)|\leqslant \beta_G(u)^{10}$.
\end{lemma}

\begin{proof}  By Theorem \ref{atleast8}, it suffices to consider groups $G$ such that $|\Aut(G)|\geqslant 8^{10}$.  These groups and $M_{11}$ appear in the table below.  For some of these groups, we obtain the necessary bound by comparison with a subgroup.  In the remaining cases,
there is a maximal subgroup $M$ which is the normalizer of a minimal non-soluble subgroup $L$.
In the table we give estimates for the maximal order of $\rC_G(u)$ for an arbitrary involution in $G$, and for $\rC_G(u)\cap\rN_G(L)$ with $u\in L$.  Our information is derived from  R.A.\ Wilson's paper \cite{robmax} and the Atlas \cite{atlas}.  We couple this with the estimate $\beta_L(u)\geqslant8$ from Theorem \ref{atleast8} and obtain the conclusion from Lemma \ref{generalestimate}. 
\end{proof}

\medskip
{\small
\begin{center}
\begin{tabular}{lccccc}
\hline \hline 
$G$ &  $|\Aut(G)|^{1/10}\!\leqslant$ &   $|\rC_G(u)|\!\geqslant$ & Max.\ Subgroup & $|\rC_{\rN(L)}(u)|\!\!\leqslant$& $\beta_G(u)\!\!\geqslant$\\ 
\hline 

$M_{11}$&  & $48$ &  $S_5$ & $8$  & 48
  \\

  \\

$He$& $10$ & $10^4$ &  $7{:}3\times L_3(2)$ & $168$  & $400$ 
  \\ 
$Ru$& $20$ & $10^5$ &  $5{:}4\times A_5$ & $80$ & $10^4$  
  \\ 
$Suz$& $20$ & $10^7$ &  $(A_6\times A_5)\cdot2$ & $8|A_6|$  & $2000$ 
  \\ 
 $Co_3$& $20$ & $10^5$ &  $A_4\times S_5$ & $96$  & $1000$ 
  \\  
$Co_2$& $30$ & $10^6$  & $(M_{11}\leqslant Co_2)$ && $48$
  \\ 
 $O'N$& $20$ & $10^5$ & $(M_{11}\leqslant O'N)$ & & $48$
  \\   
$Fi_{22}$& $30$ & $10^6$ &  $(M_{11}\leqslant   Fi_{22})$ &&$48$
  \\ 
 $HN$& $40$ & $10^6$ &  $(M_{11}\leqslant HN)$ && $48$
  \\  
 $Ly$& $60$ &  & $2\cdot A_{11}$ &cf.\ Lemma 3& $24^2$ 
  \\ 
$Th$& $60$ & $10^7$ &  $S_5$ & $8$ & $10^7$ 
  \\ 
  $Fi_{23}$& $100$ & $10^8$ &  $L_2(23)$ & $24$ & $4\times 10^6$ 
  \\
  $Co_1$& $100$ & $10^8$ &  $(A_5\times J_2)\colon2$ & $8|J_2|$ & $140$ 
  \\ 
  $J_4$& $100$ & $10^9$ &  $L_2(23)\colon2$ & $48$ & $10^7$ 
  \\ 
  $Fi_{24}'$ & $300$ & $10^{11}$ &  $(Fi_{23}\leqslant Fi_{24}')$ && $4\times 10^6$ 
  \\
  $B$& $300$ & $10^{16}$ &  $L_2(17)\colon2$ & $32$ & $10^{14}$ 
  \\  
  $M$& $10^{6}$ & $10^{26}$ &  $(2\cdot B\leqslant M)$ & $40$ &$10^{14}$ 
  \\ 
  
 \hline \hline 
\end{tabular} 
\end{center}}

At least for the sporadic groups it should be possible to improve the above  bound: almost certainly
$|\Aut(G)|\leqslant \beta_G(u)^{4}$ for all sporadic groups.  It may be that a similar bound holds
for all simple groups, but to establish this would require more computation and other methods.

\paragraph{Acknowledgements} The first author was supported by a CMS/SRIM bursary from the University of Cambridge.

\bigskip

\noindent \textsc{David Popović}\\
\textsc{Churchill College, Cambridge, UK}\\
\href{mailto:dp574@cam.ac.uk}{dp574@cam.ac.uk}

\bigskip

\noindent \textsc{John S. Wilson}\\
\textsc{Christ's College, Cambridge, UK}\\
\href{mailto:jsw13@cam.ac.uk}{jsw13@cam.ac.uk}

and

\noindent \textsc{Universit\"at Leipzig}\\
\textsc{Mathematisches Institut}\\
%Augustusplatz 10
\textsc{04109 Leipzig, Deutschland}

	\end{document}